\documentclass{article}

\usepackage[english]{babel}

\usepackage[letterpaper,top=2cm,bottom=2cm,left=3cm,right=3cm,marginparwidth=1.75cm]{geometry}

\usepackage{amsmath}
\usepackage{amssymb}
\usepackage{amsthm}
\usepackage{graphicx}
\usepackage[dvipsnames]{xcolor}
\usepackage[colorlinks=true, allcolors=blue]{hyperref}

\newcommand{\bb}[1]{\mathbb{#1}}
\newcommand{\Op}{\text{Op}}
\newcommand{\Trace}{\operatorname{Trace}}

\newcommand{\Vol}{\text{Vol}}
\newcommand{\dist}{\text{dist}}

\newtheorem{theorem}{Theorem}

\newtheorem{lemma}[theorem]{Lemma}
\newtheorem{corollary}[theorem]{Corollary}
\theoremstyle{remark}\newtheorem*{remark}{Remark}
\theoremstyle{definition}\newtheorem*{definition}{Definition}

\def\ang#1{{\langle #1 \rangle }}
\def\RR{\mathbb{R}}
\def\CC{\mathbb{C}}
\def\HH{\mathbb{H}}

\def\NN{\mathbb{N}}

\def\LL{\mathcal{L}}

\definecolor{olivesample}{rgb}{0.0,0.5,0.0}

\newcommand\PSL{\mathrm{PSL}(2, \RR)}
\newcommand\fT{\frak{T}}
\renewcommand\H{\mathcal{H}}

\title{Geodesic flow and decay of traces on hyperbolic surfaces}
\author{Antoine Gansemer and Andrew Hassell}

\begin{document}
\maketitle

\begin{abstract}
We study pseudodifferential operators on a hyperbolic surface using `Zelditch quantization' \cite{zel.1}. We motivate and study the trace of $A_2^* A_1(t)$, where $A_2$ is a fixed operator and the Zelditch symbol of $A_1(t)$ evolves by geodesic flow. We find conditions under which the trace decays exponentially as $t \to \pm \infty$. 
%We explore a relationship between two kinds of correlations on hyperbolic space. One is a classical correlation and the other is a quantum correlation.  We explore if knowledge of one of the correlations can give us information about the other correlation. We also prove a ``decay of semi-quantum correlations" result.
\end{abstract}

\section{Introduction: Background and motivation}
 In this paper we consider traces of time-varying families of pseudodifferential operators defined on a compact hyperbolic surface. Our interest in this topic arises from the \emph{exact} intertwining, discovered by Anantharaman and Zelditch \cite{ana.zel.2},  between the classical flow on symbols of such operators (geodesic flow acting on symbols via pullback) and the quantum flow (conjugation by the Schr\"odinger group $e^{it\Delta}$ acting on the associated operators). Of course an exact intertwining only makes sense if there is an exact correspondence between symbols and operators. On  hyperbolic space this is provided by the Zelditch calculus \cite{zel.1}, which has the key property of left-invariance, i.e.\ if a symbol is invariant under the left action of a discrete group $\Gamma$, then the associated operator is also $\Gamma$-invariant. Thus the Zelditch quantization descends to compact hyperbolic surfaces. 

Our main result in this note actually does not involve the Anantharaman-Zelditch intertwining operator $\LL$ at all. Nonetheless, the intertwining operator provides the motivation for posing the question that we answer here, and we expect that the result presented here will be an ingredient of a larger program in which the intertwining operator plays a leading role. Therefore we explain some of this background and motivation before describing the particular result we present here. 

In the study of eigenvalues and eigenfunctions related to elliptic operators on compact manifolds, it is natural to consider frequency intervals of approximately unit length, independent of frequency; that is, we look at a frequency interval $[\lambda, \lambda + 1]$. For example, H\"ormander in his seminal article \cite{hor.spectral.fn} obtained a $O(\lambda^{n-1})$ kernel bound on the spectral projection $E_{[\lambda, \lambda + 1]}(\sqrt{\Delta})$ where $\Delta$ is the Laplacian on a compact manifold without boundary\footnote{Our sign convention is that $\Delta$ is a positive operator}. This implies the optimal $O(\lambda^{n-1})$ bound on the remainder term in the Weyl asymptotic formula for the eigenvalue counting function $N(\lambda)$, the rank of the spectral projection $E_{[0, \lambda]}(\sqrt{\Delta})$, and also implies an $O(\lambda^{(n-1)/2})$ bound on the $L^\infty$ norm of eigenfunctions with eigenvalue $\lambda^2$ (or indeed, on the $L^\infty$ norm of any spectral cluster with support in $[\lambda, \lambda + 1]$). This analysis is achieved by passing to an approximate spectral projection where the indicator function $1_{[\lambda, \lambda + 1]}$ is smoothed to a Schwartz function $\rho(\cdot - \lambda)$, such that the Fourier transform $\hat \rho(t)$ has compact support in $t$, and then representing $\rho(\sqrt{\Delta} - \lambda)$ in terms of the half-wave group $e^{it \sqrt{\Delta}}$:
\begin{equation*}
\rho(\sqrt{\Delta} - \lambda) = \frac1{2\pi} \int e^{it \sqrt{\Delta}} \hat \rho(t) \, dt.
\end{equation*}
One then analyzes the half-wave group for $t$ in the support of $\hat \rho$, which can be taken to be a small interval containing zero. The key then is to analyze the singularity of the (distributional) trace of the half-wave kernel as $t \to 0$. 

To do the analysis at smaller frequency scales, that is, on smaller than unit-size intervals as $\lambda \to \infty$, we scale $\rho$ so that it concentrates near 0 as $\lambda \to \infty$; dually, $\hat \rho$ will scale so as to have larger support. We then hit new singularities of the wave trace, which occur at the length spectrum of $\sqrt{\Delta}$; that is, at the lengths of periodic bicharacteristics of the operator (geodesics, in the case of the Laplacian). To obtain an improvement on the remainder term in the Weyl asymptotic formula to $o(\lambda^{n-1})$, or the corresponding $o(\lambda^{(n-1)/2})$ improvement in the $L^\infty$ bound on eigenfunctions, one studies the wave group for arbitrarily long times, and (necessarily, in view of the example of the round sphere) requires a condition  on the dynamics of bicharacteristic flow, for example, that the set of periodic bicharacteristics is measure zero in the characteristic variety \cite{duist.guille, ivrii}. 

To obtain a quantitative improvement on the Weyl remainder, say  to $O(\lambda^{n-1}/\log \lambda)$, one needs to narrow the spectral window to $[\lambda, \lambda + A/\log \lambda]$. Then, taking account of the  dual scaling between the function $\rho$ and its Fourier transform $\hat \rho$, this requires looking at the wave trace for a time scale that is logarithmic in the frequency. This can be done in the case of the Laplacian on manifolds with negative curvature (or, in the case of two dimensions, manifolds without conjugate points) \cite{Berard}. See also \cite{hassell.tacy.1} and a series of works by Sogge and co-authors, for example \cite{sogge.17}, \cite{blair.sogge} for further results on spectral clusters associated to logarithmically narrowed spectral windows. 

To go beyond this point, i.e.\ to consider even narrower spectral windows, seems very difficult in an arbitrary geometric situation due to the restriction in Egorov's Theorem. Egorov relates the `quantum flow', that is, the half-wave group and the `classical flow', that is, the geodesic flow on the cosphere bundle, but in an arbitrary geometry it is typically only valid up to a time proportional to the logarithm of the frequency, a time scale which is already saturated in the articles cited in the previous paragraph. However, one might still expect that it is possible to go to finer scales in frequency in very special geometries, such as constant negative curvature. The first interesting case is that of compact hyperbolic surfaces.

This, then is the (potential) significance of the Anantharaman-Zelditch intertwining operator: as an exact intertwining between the classical geodesic flow and the Schr\"odinger group (which in spectral analysis can play the same role as the half-wave group discussed above), it is in particular valid for all times and thereby avoids the `Egorov restriction' to times smaller than the logarithm of the frequency. Thus it provides a potential pathway to considering much smaller spectral windows and, hence, much more precise Weyl remainder terms, more precise $L^\infty$ bounds on eigenfunctions, more precise small-scale quantum ergodicity, and so on.

In the study of small-scale quantum ergodicity, say on compact hyperbolic surfaces to connect with the work of Anantharaman-Zelditch,  one is led to the study of time-varying families of traces of the form 
\begin{equation}
\Trace A_h e^{iht \Delta/2} e^{-h^2 \Delta} A_h e^{-ith\Delta/2} e^{-h^2 \Delta},
\label{eq:qc}\end{equation} %\ah{I put a factor of two into the Schrodinger group to conform to AZ.}
where $h$ is a semiclassical parameter, representing inverse frequency, and $A_h$ is a multiplication operator by a function with support in a ball around a fixed point with a radius tending to zero as $h \to 0$. Here the factors of $e^{-h^2 \Delta}$ are a soft frequency cutoff to ensure that the composition is trace class. See the works of Han \cite{han}, and Hezari-Rivi\`{e}re \cite {hez.riv.1}, for more on small scale quantum ergodicity at logarithmic length scales on negatively curved manifolds, the work of Zelditch \cite{zel.2} and Schubert \cite{schu} are also closely related.
We use the Schr\"odinger propagator rather than the half-wave group to conform with the analysis in Anantharaman-Zelditch. (The bicharacteristic flow of the Schr\"odinger propagator is regular at zero frequency, which is not the case for the half-wave propagator. This was important for Anantharaman and Zelditch as they were interested in exact formulae, for which one cannot simply cut off at low frequencies.) More generally, we consider traces of the form 
\begin{equation}
 h^2 \Trace B_h e^{iht \Delta/2} A_h e^{-ith\Delta/2} ,
\label{eq:traceBA}\end{equation}
where $A_h$ and $B_h$ are pseudodifferential operators of differential order $-\infty$ and semiclassical order $0$ on a compact hyperbolic surface. The factor of $h^2$ in front is to cancel the growth rate of the trace of an operator of semiclassical order zero as $h \to 0$, which is $\sim h^{-2}$ in two dimensions. 
One can ask about the asymptotic property of such a trace for large time $t$. In particular, it would be useful to understand under what conditions \eqref{eq:traceBA} tends to zero as $t \to \infty$, and if so at what rate. 

We now consider the following \emph{purely formal} calculation. We write the trace in \eqref{eq:traceBA} as a bilinear functional acting on the Zelditch symbols $a$ and $b$, which we denote $\ang{\cdot, \cdot}$. Notice that on the whole of hyperbolic space, the trace would be the integral of $ab$ over the cotangent bundle, but this is not the case on a hyperbolic surface. Instead it takes the form 
\begin{equation}\label{eq:traceform}
\Trace B^* A = \int a(g, r) \overline{\fT b (g, r)} \, dg \, dr
\end{equation}
where $\fT$ is a linear operator mapping symbols to distributions, such that $ \fT b$ has support on the countable number of energy shells of radius $\lambda_j$, where the spectrum of $\Delta_X$ is $\{ 1/4 + \lambda_j^2 \}$. This identity is derived in Lemma~\ref{lemma.trace.formula}.

We use the notation from \cite{ana.zel.2}: $V^t$ is the action induced on the Zelditch symbol by conjugation with the Schrodinger propagator $e^{it \Delta}$ at time $t$. Thus the symbol of the operator $e^{iht \Delta/2} A e^{-ith\Delta/2}$ from \eqref{eq:qc}  is $V^{th} a$. So we find that 
\begin{equation}
\Trace B^* e^{ith \Delta/2} A e^{-ith\Delta/2} =  \int (V^{th} a)(g, r) \overline{(\fT b)(g, r)} \, dg \, dr .
\end{equation}
Let $\LL$ denote the Anantharaman-Zelditch intertwining operator. In fact, in unpublished work, the first author has found that  a slight adjustment of the definition of the intertwining operator in \cite{ana.zel.2} leads to  $\LL$ being an isometry on $L^2$. Under some regularity assumption on $\LL$, one would have (and this is where the argument becomes formal)
\begin{equation}
\Trace B^* e^{ith \Delta/2} A e^{-ith\Delta/2} =  \int (\LL V^{th} a)(g, r) \overline{\LL \big( \fT b (g, r) \big)} \, dg \, dr .
\end{equation}
We now apply the intertwining property of $\LL$ (see \cite[Section 1.4]{ana.zel.2}), namely
\begin{equation}
\LL V^t  = G^t \LL,
\label{eq:intertwining}\end{equation}
where $G_t$ is the geodesic flow scaled by the speed factor $r$ on the energy shell of radius $r$ (this is the bicharacteristic flow for the operator $e^{it\Delta/2}$). Thus we obtain 
\begin{equation}\label{eq:classicalflow}
\Trace B^* e^{iht \Delta/2} A e^{-ith\Delta/2} = \int (G^{th} \LL  a)(g, r) \overline{\LL \big( \fT b (g, r) \big)} \, dg \, dr .
\end{equation}
Now comparing this identity with \eqref{eq:traceform}, we see that this is analogous to studying the trace of $B^* A(t)$ where $B$ is held fixed and the symbol of $A$ evolves according to geodesic flow. This is exactly the question we study here: when $B$ is held fixed and the symbol of $A$ evolves according to geodesic flow, under what additional conditions can we deduce that the trace of $B^* A(t)$ decays, and what is the rate of decay? Our main result, Theorem~\ref{theorem.main}, is an answer to this question. 
%The RHS is of a form where we can hope to apply decay of correlations, leading to exponential decay of the trace. Of course, this would require that the two factors, $\LL a$ and $\LL Tb$ are in appropriate function spaces. 
%
%Hence this program requires two main steps:
%
%\begin{itemize}
%
%\item We need to show that 
%find appropriate function spaces $\XX$, $\YY$ such that, for a suitable interesting class of symbols $a$ and $b$, the intertwining operator $\LL$ maps into these function spaces. The function spaces must be such that decay of correlations holds for one function in $\XX$ and one in $\YY$. 
%
%\item We need to understand the trace of $B^* A(t)$ where $B$ is held fixed, and the symbol of $A(t)$ evolves by the classical (geodesic) flow. We need to show decay of the trace for suitable choices of $a$ and $b$. 
%\end{itemize}
%
%In this article we achieve the second step. The first is under active investigation by the authors. 

\section{Preliminaries and notation}
\subsection{Hyperbolic space, metric, volume form, dynamics}
Let us fix the upper half plane, $\bb{H}^2$ as the set $\{z=x+iy: x\in\bb{R}, y>0\}$ and equip it with the standard smooth (holomorphic) structure considered as an open subset of $\bb{C}$. Further to this, consider the standard hyperbolic metric on $\bb{H}^2$, given by
\[
ds^2=\frac{dx^2+dy^2}{y^2}.
\]
This metric induces the following volume form and (positive) Laplace-Beltrami operator respectively:
\[
d\Vol = \frac{dxdy}{y^2}, \qquad \Delta_{\bb{H}^2}=y^2\left( D_x^2 + D_y^2 \right), \quad D_x = \frac1{i} \frac{\partial}{\partial x}, \quad D_y = \frac1{i} \frac{\partial}{\partial y}.
\]
We can map $\bb{H}^2$ bijectively into the unit disk $\bb{D}=\{z=x+iy: |z|<1\}$ using the Cayley transform,
\begin{equation}\label{cayley.transform}
z \mapsto \frac{z-i}{z+i}.
\end{equation}
We define this map to be an smooth isometry between $\bb{H}^2$ and $\bb{D}$, inducing the hyperbolic metric on the unit disk 
\begin{equation}\label{hyperbolic.disk.metric}
ds^2=4\frac{dx^2+dy^2}{(1-x^2-y^2)^2}.
\end{equation}
%This also places a smooth structure on $\bb{D}$, which is equivalent to the standard smooth structure on $\bb{D}$ considered as an open subset of $\bb{C}$. 
The map extends smoothly to the boundary of $\bb{H}^2$, the set $\{x+iy:y=0\}$, which is mapped smoothly to  $\partial \bb{D} \setminus \{ 1 \}$. We  introduce a point at infinity in the upper half-plane model which maps to $1\in\partial\bb{D}$ to make this a bijection. 

Taking a look now at the isometry group of $\bb{H}^2$, it is well-known that the fractional linear action of $\PSL$ on $\bb{H}^2$
\[
\begin{pmatrix}
    a & b\\
    c& d
\end{pmatrix}: z \mapsto \frac{az+b}{cz+d},
\]
is an isometry of $\bb{H}^2$.  The isometry group of the disk model is $\text{PSU}(1,1)$ which is conjugate to $\PSL$ by the Cayley transform. We write $G$ for $\PSL$ or $\text{PSU}(1,1)$, depending on the model of hyperbolic space we are using (and we will often move between the two models freely), and let $g$ or $\gamma$ denote an element of $G$. It is again easy to check that $\PSL$ acts transitively and faithfully on $S\bb{H}^2$, hence we can identify $S\bb{H}^2$ with $G=\PSL$ with the following identification
\[
\PSL\ni g \mapsto (g(i), g'(i)i)\in S\bb{H}^2.
\]
In particular this identifies the identity element of $G$ with the unit vector $(i,i)$ based at $i \in \bb{H}^2$ pointing along the imaginary axis. With this identification of  $G$ with $S\bb{H}^2$, the geodesic flow starting at a point $h\in G$ after time $t$ is denoted by $g^t(h)$ and is given by the right action, $g^t(h)=h\mathbf{a}_t$, where $\mathbf{a}_t$ is the element of $G$ given by 
\[
\mathbf{a}_t:= \begin{pmatrix}
e^{t/2} & 0 \\
0 & e^{-t/2}
\end{pmatrix}.
\]
The geodesics on $\bb{H}^2$ or $\bb{D}$ trace out circular arcs orthogonal to the boundary at conformal infinity. This fact is well known in hyperbolic geometry.  See, for example, \cite{helg.1} for a proof of this fact.

We define the stable horocyclic flow and the unstable horocyclic flow on $G$ by the right action with the elements $\mathbf{n}_u$ and $\overline{\mathbf{n}_v}$ respectively, where
\[
\mathbf{n}_u:=\begin{pmatrix}
1 & u\\
0 & 1
\end{pmatrix},
\qquad
\overline{\mathbf{n}_v}:=\begin{pmatrix}
1 & 0\\
v & 1
\end{pmatrix}.
\]  

We also define the element $\mathbf{k}_\theta$, given by 
\begin{equation}
\mathbf{k}_\theta = \begin{pmatrix}
\cos \frac{\theta}{2} & \sin \frac{\theta}{2}  \\
-\sin \frac{\theta}{2}  & \cos \frac{\theta}{2}
\end{pmatrix}, \qquad \theta\in [0,2\pi),
\end{equation}
whose right action fixes the base point $z$ and rotates by angle $\theta$ in the fibre of $S \bb{H}^2$. Notice that $\mathbf{k}_{2\pi} = \mathbf{k}_0$ in $\PSL$.  We can define global coordinates $(z, \theta)$ on $G \simeq S \bb{D}$ by taking $(z, \theta = 0)$ to be the unit vector such that the geodesic emanating from this vector tends to a fixed point, say $1$, on $\partial \bb{D}$, and so that the action of $\mathbf{k}_\theta$ acts by $(z, \theta') \mapsto (z, \theta' + \theta)$. 

The group $G$ has a bi-invariant Haar measure, which we denote $dg$, given by $d\Vol(z) d\theta$ in $(z, \theta)$ coordinates. 

The elements $\mathbf{a}_t$ form a subgroup of $G$ denoted $A$. Similarly the elements $\mathbf{n}_u$, $\overline{\mathbf{n}_v}$ and $\mathbf{k}_\theta$ form subgroups, denoted $N$, $\overline{N}$ and $K$ respectively. We define the vector fields $H$, $X_{+}$, $X_{-}$ on $G\simeq S\bb{H}^2$ as the infinitesimal generators of the subgroups $A$, $N$ and $\overline{N}$ respectively; that is, they generate the geodesic flow, the stable horocyclic flow and the unstable horocyclic flow respectively.  Let us note that $H$, $X_{+}$ and $X_{-}$ are left-invariant vector fields on $G$ because the flows they generate act as right actions on $G$, and hence commute with the left action of $G$ on itself.  The Lie brackets of these vector fields satisfy:
\begin{equation}
[H,X_{+}]=X_{+}, \quad [H, X_{-}] =-X_{-}, \quad [X_{+},X_{-}]=2H.
\end{equation}
which are easily derived from the matrix expressions for $\mathbf{a}_t, \mathbf{n}_u,$ and $ \overline{\mathbf{n}_v}$.

We can visualise these three flows on $G$ using the identification $G\simeq S\bb{D}$. Please see Figure \ref{fig.flows}.

Let us also introduce coordinates $(z,b)\in\bb{D}\times\bb{S}^1$. Given $(z,v)\in S\bb{D}$, there is a unique $b\in\bb{S}^1$ such that the geodesic flow applied to $(z,v)$, $g^t(z,v)$ has forward (conformal) endpoint $b\in\bb{S}^1$ as $t\rightarrow\infty$. In this way the fibre $S_z\bb{D}$ and $\bb{S}^1$ are diffeomorphic (in a $z$-dependent way). Hence we have identifications $G\simeq S\bb{D}\simeq \bb{D}\times\bb{S}^1$.

%The Haar measure on $G$, which is both left- and right-invariant, we denote by $dg$. It is given by 
%\begin{equation}
%dg = d\Vol(z) d\theta = 

Fix a point $z\in\bb{D}$ and $b\in\bb{S}^1$. We define the \emph{horocycle} through $z$ and tangent to $b$ as the limit of a family of curves through $z$. Explicitly, consider the curve 
\[
C(z,\beta):=\{z'\in\bb{D}: \dist_{\bb{D}}(z',\beta)=\dist_{\bb{D}}(z,\beta)\}
\]
where $\dist_{\bb{D}}(z',\beta)$ refers to the distance on $\bb{D}$ between points $z'$ and $\beta$ induced by the hyperbolic metric, (\ref{hyperbolic.disk.metric}).  These are circles in the hyperbolic disk with centre point $\beta$ passing through $z$. As $\beta$ tends towards the point $b\in\bb{S}^1$ at infinity, say along a geodesic, the curves $C(z, \beta)$ converge smoothly to a curve that we call the horocycle through $z$ tangent to $b$.
%, as the limit of the locus $C(z,\beta)$ as $\beta\rightarrow b$ (the limit as $\beta\rightarrow b$ is taken in the Euclidean metric on $\bb{D}$, conformal to the hyperbolic metric). 
In the disk model, a horocycle is a Euclidean circle passing through $z$ and tangent to the boundary, $\bb{S}^1$ of the disk at $b$.  See Figure \ref{fig.flows}.  For comparison, the corresponding sets in Euclidean space $\RR^n$ relative to a point $\omega$ on the sphere at infinity are hyperplanes with normal vector $\omega$.

\begin{figure}[h]
\caption{By identifying $g$ with a point in $S\bb{D}$, the \textcolor{olivesample}{geodesic flow}, \textcolor{blue}{stable horocyclic flow}, \textcolor{red}{unstable horocyclic flow} and \textcolor{orange}{rotation in the fibres} are represented by the right action on $g$ by $\mathbf{a}_t$, $\mathbf{n}_u$, $\overline{\mathbf{n}_v}$ and $\mathbf{k}_{\theta}$ respectively. The images of the two horocyclic flows trace out horocycles in the base point $z$, which (in the disk model) are circles tangent to the conformal boundary at the forward and backward endpoints ($b$ and $b'$ respectively) of the geodesic determined by $g$.}\label{fig.flows}.
\includegraphics[scale=1.3]{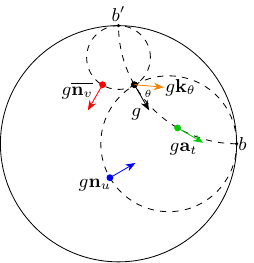}
\centering
\end{figure}

%\begin{lemma}[KAN decomposition]
%Every $g\in G$ can be written uniquely as the matrix product of an element $k_{\theta}\in PSO(2)$, an $\mathbf{a}_t$ and an $\mathbf{n}_u$ (as defined above). Note that the family, $\mathbf{a}_t$, over $t\in\bb{R}$ and the family, $\mathbf{n}_u$, over $u\in\bb{R}$ form subgroups of $G$, denoted $A$ and $N$ respectively.
%\end{lemma}
%\ag{To complete: Prove this lemma}
%\ah{We should define $dg$ to be the Haar measure and notice that $dg=dvol(z)d\theta=e^{\langle z, b\rangle}dvol(z)db$}

Let us now state the definition of an Anosov flow and present a short direct proof that the geodesic flow on $S\bb{H}^2$ is Anosov.
\begin{definition}[Anosov flow]
A smooth flow $\varphi_t$ on a smooth manifold $M$ is called \emph{uniformly hyperbolic} or \emph{Anosov} if, for all $x\in M$, there exist a splitting $T_xM=E_0(x)\oplus E_u(x)\oplus E_s(x)$ which is preserved by the flow, i.e. $E_i(\varphi_t(x))=d\varphi_t(x)(E_i(x))$ for all $x\in M$, $t\in\bb{R}$ and $i = 0, u, s$, where $E_0(x)$ is spanned by the vector field which generates $\varphi_t$ at $x$, and such that, for any fixed norm on the fibres $T_xM$, there exists constants $C>0$, $\lambda>0$ such that
\begin{equation}
    \begin{split}
        |d\varphi_t(x)(v)|\leq Ce^{-\lambda t}|v|, \quad \text{if} \quad v\in E_s(x), t>0\\
        |d\varphi_t(x)(v)|\leq Ce^{-\lambda t}|v|, \quad \text{if} \quad v\in E_u(x), t<0
    \end{split}
\end{equation}
Consequently, $E_s(x)$ is called the stable subspace of $T_xM$ and $E_u(x)$ is called the unstable subspace of $T_xM$. 
\end{definition}

\begin{lemma}[geodesic flow on $S\bb{H}^2$ is an Anosov flow]\label{lemma.geodesic.is.anosov}
The geodesic flow on $S\bb{H}^2$ is an Anosov flow with stable and unstable leaves given by the stable and unstable horocycles respectively.
\end{lemma}
\begin{proof}
    Fix a point $h\in G\simeq S\bb{H}^2$ and consider the geodesic flow $g^t(h)=h\mathbf{a}_t$. Consider the two leaves through $h$ given by the unstable and stable horocycle flows respectively, $u\mapsto h\mathbf{n}_u$ and $u\mapsto h\overline{\mathbf{n}_u}$. These leaves are generated by $X_+$ and $X_{-}$ respectively and we claim the line bundles over $h$ spanned by $X_{+}$ and $X_{-}$ are $E_s(h)$ and $E_u(h)$ respectively. Let us fix a left-invariant norm on the fibres of $TG$ by setting $|H|=|X_+|=|X_-|=1$. We can make the explicit (matrix) computation that $\mathbf{n}_u\mathbf{a}_t=\mathbf{a}_tn_{ue^{-t}}$ and $\overline{\mathbf{n}_u}\mathbf{a}_t=\mathbf{a}_t\overline{n_{ue^t}}$ which shows that
    \[
    dg^t(h)X_+=e^{-t}X_+, \qquad dg^t(h)X_-=e^{t}X_-.
    \]
    This shows that the line bundles $E_s(h)$ and $E_u(h)$ are preserved by the $g^t$ flow and that $X_+$ spans the stable sub-bundle and $X_-$ spans the unstable sub-bundle with Lyapunov exponent $\lambda=1$. Hence the geodesic flow $g^t$ on $G\simeq S\bb{H}^2$ is Anosov.
\end{proof}

The splitting of the tangent bundle into stable, unstable and flow directions induces a splitting of the cotangent bundle. These bundles are denoted $E_s^*, E_u^*$ and $E_0^*$ and are defined by 
\begin{equation}\begin{aligned}
E_s^* &= (E_s \oplus E_0)^0 \\
E_u^* &= (E_u \oplus E_0)^0 \\
E_0^* &= (E_s \oplus E_u)^0
\end{aligned}\end{equation}
where the superscript $0$ indicates the annihilator space. These spaces can also be defined by the following equations, where $h, u, s$ are the symbols of $H, X_+, X_-$ respectively:
\begin{equation}\begin{aligned}
E_s^* &= \{ h = u = 0 \} \\
E_u^* &= \{ h = s = 0 \} \\
E_0^* &= \{ s = u = 0 \}
\end{aligned}\end{equation}
Then under the bicharacteristic flow of $H$ (that is, the geodesic flow lifted to the cotangent bundle of $S\bb{H}^2$), we have $\dot u = u$ and $\dot s = -s$, so the flow is expanding on $E_u^*$ and contracting on $E_s^*$, which explains the notation for these line bundles. This splitting of the cotangent bundle is important in the definition of anisotropic Sobolev spaces (see Section~\ref{subsec:aniso}).

\subsection{Harmonic analysis on $\bb{D}$}
\begin{definition}[Busemann function]
Given $z \in \bb{D}$ and $b \in \bb{S}^1$. The Buseman function of $z$ and $b$, denoted $\langle z, b \rangle$, is defined to be the signed distance from $0 \in \bb{D}$ to the horocycle through $z$ tangent to $b$, with sign convention that $\ang{z, b}$ tends to infinity as $z \to b$ along any geodesic. 
\end{definition}
The family of geodesics emanating from $b\in \bb{S}^1$ are orthogonal to the family of horocycles tangent to $b$. The identity $\mathbf{n}_u\mathbf{a}_t=\mathbf{a}_t\mathbf{n}_{ue^{-t}}$ has the geometric interpretation that the signed distance between horocycles is well-defined, as any geodesic segment between the two horocycles has the same length $t$. 
%This means that any two horocycles tangent to the same point $b\in\bb{S}^1$ have a well-defined notion of distance between one another.  So for any $z,w\in\bb{D}$, we have that for all $z'\in H(z,b)$, $\text{dist}_{\bb{D}}(z, H(w,b)) = \text{dist}_{\bb{D}}(z',H(w,b))$ and similarly, for all $w'\in H(w,b)$, $\text{dist}_{\bb{D}}(w',H(z,b)) = \text{dist}_{\bb{D}}(w, H(z,b))$. This means the signed distance between the two horocycles $H(z,b)$ and $H(w,b)$ is a well-defined quantity. 
In particular it is the quantity given by $\langle z , b\rangle - \langle w, b\rangle$ where $z$ is any point on the first horocycle and $w$ is any point on the other horocycle. See Figure \ref{fig1}.

The function $e^{\ang{z, b}}$ coincides with the classical Poisson kernel on the disk, 
\begin{equation}
e^{\ang{z,b}} = P(z, b) : = \frac{1 - |z|^2}{|z-b|^2}.
\label{eq:Poisson}\end{equation}
In fact, by direct calculation one can verify that this function is constant on horocycles based at $b$, it equals $1$ when $z=0$, and the gradient $\nabla_z \langle z, b\rangle$, in the hyperbolic metric, has unit length. Using this fact about the gradient and the harmonicity of $P(z, b)$ (notice that harmonicity coincides for the Euclidean and the hyperbolic metric), it is easy to check that for $\mu \in \bb{C}$ and $b \in \partial \bb{D}$, the function $z \mapsto e^{\mu \langle z, b \rangle}$ is an eigenfunction of the hyperbolic Laplacian with eigenvalue $\mu (\mu-2)$.  We refer to these powers of the Poisson kernel $e^{\ang{z, b}}$ as (hyperbolic) plane waves, as they are analogous to Euclidean plane waves. Helgason used these hyperbolic plane waves to define a non-Euclidean Fourier transform on the hyperbolic disk.

The Busemann function is also useful in linking the two coordinate systems $(z, \theta)$ and $(z, b)$ that we have discussed. In fact, for fixed $z$ we have $d\theta = e^{\ang{z, b}} db$. (The easiest way to derive this is to consider a harmonic function $u$ on the disc with boundary values $f(b)$. Then we have by the mean value theorem
$$
u(z) = \frac{1}{2\pi}\int_{\bb{S}^1} f(b(\theta)) \, d\theta = \frac{1}{2\pi}\int_{\bb{S}^1} P(z, b) f(b) \, db. )
$$
This implies that the Haar measure in $(z, b)$ coordinates is 
\begin{equation}
dg = d\Vol(z) d\theta = e^{\ang{z, b}} d\Vol(z) db.
\label{eq:haar}\end{equation}

\begin{figure}[t]
\caption{On the left, a \textcolor{blue}{horocycle} through $z$ and $b$ and a \textcolor{blue}{horocycle} through $w$ and $b$ on the hyperbolic disk. The so-called Busemann function, $\langle z, b\rangle$, is the \textcolor{red}{distance} from the origin $o$ to this horocycle. $\langle z, b\rangle-\langle w, b\rangle$ represents the signed distance between the given horocycles. On the right, a family of horocycles through a single point $b\in B$. Any two horocycles through the same point $b$ are equidistant to each other.}\label{fig1}.
\includegraphics[scale=1.5]{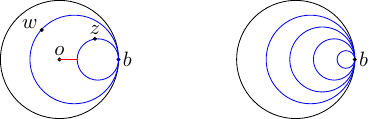}
\centering
\end{figure}

One of the important identities involving the Busemann function is how it changes under a left action of $\gamma\in G$ (acting on both $z$ and $b$, that is, using the coordinates $(z, b)$ on $G$). We note that $\ang{z,b}$ is not invariant under $\gamma$ since it depends on a choice of origin in $\bb{H}^2$. However, a difference $\ang{z, b} - \ang{w, b}$ is invariant. We thus have 
$$
\ang{\gamma z, \gamma b} - \ang{\gamma w, \gamma b} = \ang{z, b} - \ang{w, b} \text{ for all } \gamma \in G.
$$
Setting $w=0$ and noting that $\ang{0, b} = 0$ by definition, we find that for all $\gamma\in G$, we have
\begin{equation}\label{Busemann.identity}
    \langle \gamma z, \gamma b\rangle = \langle z, b\rangle + \langle \gamma 0 , \gamma b\rangle.
\end{equation}

Another useful identity involving the Busemann function is 
\begin{equation}\label{eq:jacobian.factor.on.b}
\gamma'(b) = e^{-\ang{\gamma(0), \gamma(b)}},
\end{equation}
which can be derived from \eqref{eq:Poisson} and the fact that $u \circ \gamma$ is harmonic on the disc if and only if $u$ is harmonic. We now state a result of Helgason, which defined the non-Euclidean Fourier transform, and showed some of its properties.

%We can sketch a proof of this identity. Since $\gamma$ is an isometry of $\bb{D}$, it preserves the space of horocycles on $\bb{D}$. Consequently, it maps the horocycle through $z$ and tangent to $b$ to the horocycle through $\gamma z$ tangent to $\gamma b$. It also maps the horocycle through $0$ tangent to $b$, to the horocycle through $\gamma 0$, tangent to $\gamma b$. Since $\gamma$ is an isometry, we have 
%\[
%\text{dist}_{\bb{D}}(H(\gamma z, \gamma b), H(\gamma 0, \gamma b))=\text{dist}_{\bb{D}}(H(z,b),H(0,b)).
%\]
%And since
%\[
%\text{dist}_{\bb{D}}(H(z_1,b),H(z_2,b)) = \langle z_1, b\rangle - \langle z_2, b\rangle
%\]
%for any $z_1,z_2\in\bb{D}$, and any $b\in B$, we consequently have 
%\[
%\langle \gamma z, \gamma b\rangle - \langle \gamma 0, \gamma b\rangle = \langle z, b\rangle - \langle 0 , b\rangle
%\]
%However $\langle 0, b\rangle=0$, which reduces the above equation to identity (\ref{Busemann.identity}).

%It is easy to check that for $\mu \in \bb{C}$ and $b \in \partial \bb{D}$, the function $z \mapsto e^{\mu \langle z, b \rangle}$ is an eigenfunction of the hyperbolic Laplacian with eigenvalue $\mu (2-\mu)$.  It is analogous to the Euclidean plane wave. Helgason used these hyperbolic plane waves to define a non-Euclidean Fourier transform on the hyperbolic disk.

\begin{theorem}[Helgason's non-Euclidean Fourier transform, Theorem 4.2 in \cite{helg.1}]
For any complex-valued function $f$ on $\bb{D}$, define the non-Euclidean Fourier transform by 
\[
\mathcal{F}f(b,r):=\int_{\bb{D}} f(z)e^{(\frac{1}{2}+ir)\langle z, b\rangle} d\Vol(z).
\]
for any $b\in B$ and $r\in\bb{C}$ for which this exists. If $f\in C^{\infty}_c(\bb{D})$, then there exists a pointwise inversion formula
\begin{equation}
f(z)=\int_{\bb{R}^+\times B} e^{(\frac{1}{2}+ir)\langle z, b\rangle} \mathcal{F}f(b,r) dp(r) db, \quad dp(r):=\frac{r}{2\pi}\tanh(\pi r)dr.
\label{eq:Fourinv}\end{equation}
%Additionally, the map $f\mapsto \mathcal{F}f$ is a bijection of $C^{\infty}_c(\bb{D})$ onto the space of smooth functions, $\psi$, on $B\times\bb{C}$ that is holomorphic in the second variable of uniform exponential type, meaning that there exists some $R>0$ such that for all $N>0$,
%\[
%\sup_{r\in\bb{C},b\in B} e^{-R |\text{Im} (r)|}(1+|r|)^N|\psi(b,r)|<\infty
%\]
%and which satisfies the functional equation
%\[
%\int_{B}\psi(b,r)e^{(\frac{1}{2}+ir)\langle z, b\rangle} db = \int_{B} \psi(b,-r) e^{(\frac{1}{2}-ir)\langle z, b\rangle} db.
%\]
This non-Euclidean Fourier transform extends to an isometry of $L^2(\bb{D}, d\Vol(z))$ to $L^2(\bb{R}^+\times B, dp(r) db)$, where $dp(r)$ is as in \eqref{eq:Fourinv}. % a certain Plancharel measure which takes the form $dp(r)=\frac{r}{4\pi}\tanh(\pi r)$.
\end{theorem}
We state another theorem of Helgason which will be useful to us.
\begin{theorem}[Boundary distribution of Laplacian eigenfunctions, Theorem 4.3 in \cite{helg.1}]\label{thm.helgason.eigenfunction.representation.formula}
    For any $r\in \bb{C}$, if $\varphi_r$ is a (smooth) function on $\bb{D}$ which satisfies $\Delta\varphi_r=(1/4+r^2)\varphi_r$ and grows at most exponentially in the hyperbolic distance, $d_{\bb{D}}(0,z)$, i.e $|\varphi_r(z)|\leq Ce^{cd_{\bb{D}}(0,z)}$ for every $z\in\bb{D}$ with $C,c>0$ some arbitrary constants, then there exists a distribution $T\in\mathcal{D}'(B)$ such that\footnote{In this article we take $T$ to be a distributional function, differing from the convention in \cite{ana.zel.2} where it is taken to be a distributional density. This changes the form of the transformation under $\gamma \in G$ by the Jacobian factor \eqref{eq:jacobian.factor.on.b}.}
    \[
    \varphi_r(z)=\int_Be^{(\frac{1}{2}+ir)\langle z, b\rangle}T(b) \, db
    \]
   % (The notation above refers to the distributional pairing of $T$ with $e^{(\frac{1}{2}+ir)\langle z, b\rangle}$ as a smooth function of $b$ for every fixed $z$.) 
    Moreover the distribution $T$ is unique if $1/2+ir\neq 0, -1, -2, \ldots$. Hence in this case, we may label $T$ as $T_{\varphi_r}$ since $T_{\varphi_r}$ is uniquely determined from $\varphi_r$ and vice versa. 
    \end{theorem}

A corollary to Helgason's theorem, Theorem \ref{thm.helgason.eigenfunction.representation.formula}, is information about how the distributions $T\in\mathcal{D}'(B)$ vary under a discrete group (also called a Fuschian group) of cocompact isometries $\Gamma\leq G$ when $T$ is the boundary distribution of an eigenfunction on the compact hyperbolic surface $\Gamma\backslash\bb{D}$.
\begin{corollary}\label{cor:resonance.distribution}
    Let $\varphi_r\in C^{\infty}(\Gamma\backslash\bb{D})$ satisfy $\Delta\varphi=(1/4+r^2)\varphi$ pointwise in $\bb{D}$. Then, for $1/2+ir\neq 0, -1, -2, \ldots$, there exists a unique boundary distribution, $T_{\varphi_r}\in\mathcal{D}'(B)$, such that for any $\gamma\in\Gamma$,
    \[
    T(\gamma b) = e^{(\frac{1}{2}-ir)\langle \gamma 0, \gamma b\rangle} T(b).
    \]
    In particular,
    \[
    e^{(-\frac{1}{2}+ir)\langle \gamma z, \gamma b\rangle}T(\gamma b) = e^{(-\frac{1}{2}+ir)\langle z, b\rangle}T(b)
    \]
    for every $\gamma\in\Gamma$, (using identity (\ref{Busemann.identity})), and so 
    \[
    e^{(-\frac{1}{2}+ir)\langle z, b\rangle}T(b)\in D'(\Gamma\backslash G),
    \]
  that is, it is invariant by $\Gamma$ and therefore forms a well-defined distribution over $\Gamma\backslash G$.
\end{corollary}
%%%%%%%%%%%%%%%%%%%%%%%%%%%%%%%%%%%%%%%%%%This remark below is maybe too long, but its nice to have.%%%%%%%%%%%%%%%%%%
%\begin{remark} Let us remark that for a diffeomorphism, $\sigma$ of $B$ that $T(d\sigma b)\in \mathcal{D}'(B)$ is defined by
%\begin{equation}\label{eqn.distribution.vs.density}
%\int_B f(b)T(d\sigma b) = \int_B f(\sigma^{-1}b) T(db).
%\end{equation}
%In particular, this distinguishes $T(db)$ as a distribution rather than a distribution density. We understand a distribution density to generally refer to ``the density (function) of a distribution". Recall that if we have a function $g$ which is integrable in $B$ under the measure $db$, this defines the distribution 
%\[
%f\mapsto \int_B f(b) g(b) db.
%\]
%If the density $g(b)$ is pulled back under a diffeomorphism $\sigma$, then the distribution with density function $g(\sigma b)$ is defined by
%\[
%\int_B f(b) g(\sigma b) db = \int_B f(\sigma^{-1}b) g(b) |\det d\sigma^{-1}(b)| db.
%\]
%Note how this differs from equation (\ref{eqn.distribution.vs.density}) defining the action of $T(d\sigma b)$. This is because $g$ is a distribution density whereas $T(db)$ a distribution. The action of $T(d\sigma b)$ corresponds analogously to pulling back the distribution $g(b)db$ under $\sigma$ to get the distribution $g(\sigma b) d(\sigma b)$.
%\end{remark}

\subsection{Zelditch pseudodifferential calculus on $\bb{D}$}
Zelditch uses Helgason's non-Euclidean Fourier transform to define a left-invariant pseudodifferential quantisation on $\bb{D}$, \cite{zel.1}.  %An important advantage of his quantisation is the left invariance between symbols and operators that the quantisation satisfies with respect to the group $G$.
For any operator $A: C^{\infty}(\bb{D})\mapsto C^{\infty}(\bb{D})$, Zelditch defines the \emph{complete symbol} of $A$ to be the function $a(z,b,r)\in C^{\infty}(\bb{D}\times B \times \bb{R}^+)$ by
\begin{equation}\label{eq:symbol}
Ae^{(\frac{1}{2}+ir)\langle z, b\rangle} = a(z,b,r) e^{(\frac{1}{2}+ir)\langle z, b\rangle}
\end{equation}
Given a symbol, $a\in C^{\infty}(\bb{D}\times B \times \bb{R}^+)$, the definition of the pseudodifferential operator $\Op(a)$ acting on $u\in C^{\infty}_c(\bb{D})$ is consequently,
\[
\Op(a)u(z)=\frac{1}{2\pi}\int_{\bb{D}_w \times B_b \times \bb{R}^+_r} a(z, b, r)e^{(\frac{1}{2}+ir)\langle z, b \rangle} e^{(\frac{1}{2}-ir)\langle w, b \rangle} u(w) dw db dp(r)
\]
where $dp(r)$ is the same Plancharel measure as in \eqref{eq:Fourinv}. The Fourier inversion formula shows that $Au=\Op(a)u$ for $A$ with complete symbol $a(z,b,r)$ and $u\in C^{\infty}_c(\bb{D})$, hence we have an exact correspondence between symbols $a\in C^{\infty}(\bb{D}\times B \times \bb{R}^+)$ and operators $A: C^{\infty}(\bb{D})\mapsto C^{\infty}(\bb{D})$ on the class of functions $C^{\infty}_c(\bb{D})$.

%extends to functions $u(z)\in C^{\infty}_c(\bb{D})$ by the hyperbolic Fourier transform on $\bb{D}$.

%\[
%\Op(a)u(z) := \int_{\bb{R} \times B} a(z,b,r)e^{(\frac{1}{2} + i r)\langle z, b \rangle} \mathcal{F}u(b,r) dp(r)db
%\]
%This definition of $\Op(a)$ relates the Schwartz kernel of the operator, $K(z,w)$, and the symbol in the following way
%\[
%K(z,w)=\int_{B\times\bb{R}} a(z,b,r) e^{(\frac{1}{2}+ir)\langle z, b\rangle} e^{(\frac{1}{2}-ir)\langle w, b\rangle} db dp(r),
%\]
%and if the symbol has polynomial growth in $r$, then this definition remains a well-defined oscillatory integral.

Let $\Gamma$ be a subgroup of $G$. 
The Zelditch calculus has the important property that a Zelditch quantised pseudodifferential operator, $\Op(a)$, commutes with $\Gamma$, in the sense that it commutes with the operators $T_\gamma$, $\gamma \in \Gamma$ acting by  $(T_{\gamma}f)(z)=f(\gamma z)$, if and only if the symbol $a$ is $\Gamma$-invariant in the sense 
$$a(\gamma z, \gamma b, r) = a(z, b, r) \text{ for all } \gamma \in \Gamma, (z, b) \in G, r \in \bb{R}^+.
$$
See \cite{ana.zel.2} for details. 
%preserves $\gamma$-periodicity, if and only if the symbol is left-$\gamma$-invariant.  An operator $A$ preserves $\gamma$-periodicity means that $A$ commutes with $T_{\gamma}$ where $T_{\gamma}$ is the left action by $\gamma$ on functions of $z$, i.e $T_{\gamma}f(z)=f(\gamma z)$.  After all, since
%\[
%a(\gamma z,\gamma b,r)=e^{-(\frac{1}{2}+ir)\langle \gamma z, \gamma b\rangle} (Ae^{(\frac{1}{2}+ir)\langle \cdot, \gamma b\rangle})(\gamma z),
%\]
%assuming that $AT_{\gamma}=T_{\gamma}A$ makes this expression equivalent to
%\[
%e^{-(\frac{1}{2}+ir)\langle \gamma z, \gamma b\rangle}(Ae^{(\frac{1}{2}+ir)\langle \gamma \cdot, \gamma b\rangle})(z)
%\]
%and then applying identity (\ref{Busemann.identity}) on the Busemann function reduces this expression back to $a(z,b,r)$. Conversely, if we assume $a(z,b,r)=a(\gamma z, \gamma b, r)$, then we can easily deduce that
%\[
%A(e^{(\frac{1}{2}+ir)\langle \cdot, b\rangle})(z)=A(e^{(\frac{1}{2}+ir)\langle \gamma^{-1} \cdot, b\rangle})(\gamma z)
%\]
%which shows that $A=T_{\gamma^{-1}}AT_{\gamma}$ on the class of functions $e^{(\frac{1}{2}+ir)\langle z, b\rangle}$ and hence on all of $C^{\infty}_c(\bb{D})$ via linearity and the Fourier transform.

Now suppose that $\Gamma$ be a discrete cocompact subgroup of $G$, that is, such that the quotient $\Gamma \backslash \bb{D}$ is a compact hyperbolic surface $X$. The invariance property just described shows that if $a$ is a $\Gamma$-invariant symbol, then $\Op(a)$ at least formally maps $\Gamma$-invariant functions to $\Gamma$-invariant functions, and hence induces an operator on $X$. %However,  we need to discuss precisely the class of symbols which are required to induce a well-defined pseudodifferential operators on the compact quotient $X=\Gamma\backslash\bb{D}=\Gamma\backslash G /K$.

%To extend the Zelditch calculus to operators on a compact quotient, $X=\Gamma\backslash \bb{H}^2=\Gamma\backslash G /K$, we will need to restrict ourselves to an appropriate symbol class so that the periodisation is well-defined.

Given an operator on $\bb{D}$ with Schwartz kernel $K(z,w)$ that is singular only on the diagonal and which is invariant under the left $\Gamma$-action of a discrete cocompact group, i.e $K(z,w)=K(\gamma z, \gamma w)$, if $K(z,\cdot)$ decays fast enough (so that the series below in (\ref{eq:kernel.on.quotient}) converges absolutely), we can define an induced operator on $\Gamma\backslash\bb{D}$ which has Schwartz kernel $K^{\Gamma}$ defined by the series
\begin{equation}\label{eq:kernel.on.quotient}
K^{\Gamma}(z,w)=\sum_{\gamma\in\Gamma}K(z,\gamma w).
\end{equation}
The fact that this is the appropriate induced kernel can be checked by the calculation
\begin{equation}
\begin{split}
\int_{X}K^{\Gamma}(z,w)f^{\Gamma}(w)dw &= \sum_{\gamma\in\Gamma} \int_{\bb{D}} K(z,\gamma w) f^{\Gamma}(w) dw \\
&= \sum_{\gamma\in\Gamma}\int_{\bb{D}}K(z, w)f^{\Gamma}(\gamma^{-1}w)dw = \int_{\bb{D}}K(z,w)f(w)dw
\end{split}
\end{equation}
where $f^{\Gamma}$ can be considered as the restriction of the $\Gamma$-periodic (automorphic) function $f$ to a fundamental domain $X$, such that
\[
\sum_{\gamma\in\Gamma}f^{\Gamma}(\gamma^{-1}w)=f(w).
\]
Rather than consider $f^{\Gamma}$ as a sharp cutoff of $f$ to a fundamental domain, we could also define $f^{\Gamma}(w):=\chi(w)f(w)$ where $\chi$ is a smooth function on $\bb{D}$ with the property that
\[
\sum_{\gamma\in\Gamma}\chi(\gamma^{-1}w)=1.
\]
The required decay of $K(z,w)$ away from the diagonal is implied by the symbol $a(z,b,r)$, having an analytic continuation in $r$ to a strip of suitable width, cf. Eguchi-Kowata \cite{eguchi.kowata}.

We also finally remark that through the identification of 
\[
S^*X \times \bb{R}^+ \simeq T^*X
\]
via polar coordinates in the fibres of $T^*X$, this means for $X=\Gamma\backslash\bb{D}$, we can identify 
\[
\Gamma\backslash (\bb{D}\times B)\times \bb{R}^+ \simeq T^*X
\]
since $ \Gamma\backslash (\bb{D}\times B)\simeq \Gamma\backslash G \simeq S^*X$.
Consequently, the class of smooth functions on the cotangent bundle of $X$ can be identified with the class of smooth functions in coordinates $z\in\bb{D}, b\in B, r\in\bb{R}^+$ which is left-$\Gamma$-invariant in the $(z,b)$ variables.  Hence, standard classes of symbols can be identified with symbols in the Zelditch calculus.  For example, the standard H\"{o}rmander symbol classes $S^m_{1,0}(T^*X)$ is equivalent via this identification to the left $\Gamma$-invariant class of symbols, $a\in C^{\infty}(\Gamma\backslash G\times\bb{R}^+)$, where for any $s\in\bb{N}$ and any multiindex $\alpha=(\alpha_1,\alpha_2,\alpha_3)$, there exists a constant $C_{s,\alpha}>0$ such that
\[
|(r\partial_r)^s H^{\alpha_1}X_+^{\alpha_2}X_{-}^{\alpha_3} a(g,r)|< C_{s,\alpha}(1+r)^m
\]
The work of Zelditch, \cite{zel.1}, goes into more detail of this pseudodifferential calculus, giving formulae for the composition, adjoints and commutators of such pseudodifferential operators, as well as deriving an analogue of Egorov's formula and Friedrichs symmetrization.

\subsection{Faure-Sj\"{o}strand anisotropic Sobolev spaces}\label{subsec:aniso}
Since the geodesic flow on $S^*X$ is a contact Anosov flow, the theory of anisotropic Sobolev spaces for contact Anosov flows introduced by Faure-Sj\"{o}strand, \cite{fau.sjo.1}, will be useful to us.  

We need a method to talk about the regularity of distributions appearing on $S^*X$. For this purpose, we will use the standard theory of microlocal (or semiclassical) analysis on $C^{\infty}$ compact manifolds, say as detailed in \cite{zwo.1}. Note that this is a different pseudodifferential calculus than the Zelditch calculus; we use the Zelditch calculus for the exact correspondence it brings between operators on $X$ and symbols on $T^*X\simeq S^*X\times\bb{R}^+$.

We recall Theorem 4 in Nonenmacher-Zworski, \cite{non.zwo.1}, which will suffice for our purposes. It reads, in our context as:

\begin{theorem}[Theorem 4 in \cite{non.zwo.1}]
\label{thm.non.zwo}
    Let $H$ be the generator of geodesic flow on $S^*X$. Consider $P=-iH$ as the self-adjoint operator on $L^2(S^*X, d\nu)$ with domain $\mathcal{D}(P)=\{u\in L^2(S^*X): Pu\in L^2(S^*X)\}$ where $\nu$ is the Liouville measure on $S^*X$.  The minimal asymptotic unstable expansion rate for the geodesic flow,
    \[
    \lambda_0:=\liminf_{t\rightarrow\infty}\frac{1}{t}\inf_{x\in S^*X} \log(\det dg^t|_{E^u(x)})
    \]
    equals one because the geodesic flow has constant unit Lyapunov exponent, as shown in Lemma \ref{lemma.geodesic.is.anosov}. For any $s>0$, there exists a Hilbert space $\mathcal{H}^{s\mathcal{G}}(S^*X):=e^{-s\mathcal{G}}L^2(S^*X)$ such that
    \[
    C^{\infty}(S^*X)\subset \mathcal{H}^{s\mathcal{G}}(S^*X)\subset \mathcal{D}'(S^*X)
    \]
    and the operator family $(P-z):\mathcal{D}^{s\mathcal{G}}\mapsto \mathcal{H}^{s\mathcal{G}}$ is meromorphic in the half plane $\Im z > -s$, admitting finitely many Pollicott-Ruelle resonances ( poles of the resolvent $(P-z)^{-1}$) in the strip $\Im s > -1/2+\epsilon$, for any $\epsilon>0$, including a simple pole at $z=0$ where the residue is a rank one projection operator onto the eigenspace of constants on $S^*X$. The resolvent estimate $||(P-z)^{-1}||_{\mathcal{H}^{s\mathcal{G}}\mapsto \mathcal{H}^{s\mathcal{G}}} \leq C\langle z\rangle ^N$ holds for $\Im z > 1/2$, $\Re z>C$ for some constants $C, N>0$.
\end{theorem}

We mention here that $\mathcal{G}$ is the Weyl quantisation of an escape function defined on $T^*(S^*X)$ constructed in \cite{non.zwo.1}. Consequently, $e^{-s\mathcal{G}}$ is a variable order pseudodifferential operator. The Hilbert spaces $\mathcal{H}^{sG}$ are similar to those appearing initially in Faure-Sj\"{o}strand, \cite{fau.sjo.1}. In particular, we can describe the Sobolev regularity at certain regions of fibre infinity for functions in the space $\mathcal{H}^{s\mathcal{G}}(S^*X)$.  Letting $A_1$ be a zeroth order pseudodifferential operator elliptic in an asymptotically conical neighbourhood of the line bundle $E_s^*$ over $S^*X$. We can say that there exists some $C>0$ such that
\[
C^{-1}\|A_1f\|_{H^s}\leq \|A_1f\|_{\mathcal{H}^{s\mathcal{G}}}\leq C \|A_1f\|_{H^s}
\]
where $H^s$ is the standard Sobolev space over $S^*X$ of order $s$. A similar equality of norms hold if $H^s$ is replaced by $H^{-s}$ and $A_1$ is instead elliptic of order zero in an asymptotically conic neighbourhood of $E_u^*$. We colloquially say that the spaces $\mathcal{H}^{s\mathcal{G}}$ are microlocally of order $H^s$ near $E_s^*$ and of order $H^{-s}$ near $E_u^*$.  In our proof below, where $\mathcal{G}$ is fixed, we will simply use $\mathcal{H}^s$ to refer to this scale of anisotropic Sobolev spaces.

We also mention that there is a relationship between the location of the poles of the resolvent of $P$ and the eigenvalues of the Laplacian for all compact hyperbolic manifolds as proved by Dyatlov-Faure-Guillarmou in \cite{dya.fau.gui.1}. Additionally, the resonant states of $P$ (the image of the residue of the resolvent at the poles) are precisely those distributions appearing in Corollary \ref{cor:resonance.distribution} for those $r_j$ corresponding to a Laplace eigenvalue (away from an exceptional set).

\section{Statement of the result}
Let $\Gamma \subset G$ be a discrete group such that $X = \Gamma \backslash \bb{D} = \Gamma \backslash G / K$ is a compact hyperbolic surface. We write $d\mu$ for the Riemannian measure associated to the hyperbolic metric on $X$, and write $L^2(X)$ for $L^2(X, d\mu)$. 

Let $a_1$ and $a_2$ be symbols in the H\"ormander class of type $(1, 0)$ and order $\alpha$ and $0$, respectively, where $\alpha \leq -6$. (More general symbols could be considered, and -6 is not sharp, but this is sufficient for our interests, since our applications will only require symbols of order $-\infty$.) We assume that they are both left-invariant by $\Gamma$, and have analytic continuations in $r$ to a strip of width strictly greater than $1/2$, so that their Zelditch quantizations $\Op(a_i)$ define operators $A_1^{\Gamma}$ and $A_2^{\Gamma}$ on $L^2(X)$. We also let $A_1^{\Gamma}(t)$ be a time-evolved family of such operators, satisfying these conditions uniformly in $t$, with $A_1(0) = A_1$. 
Connecting with the discussion in the first section, our long-term interest is in the family $A_1^{\Gamma}(t) = e^{-it\Delta_X/2} A_1^{\Gamma} e^{it\Delta_X/2}$, the `quantum evolution of $A_1^{\Gamma}$' in the Heisenberg picture, but here, motivated by the RHS of \eqref{eq:classicalflow}, we consider the family $\Op(a_1 \circ g^t)$ where $g^t$ is the geodesic flow on $S^* \HH^2$ (notice that these symbols are left-$\Gamma$-invariant for all $t$). Our main result is

\begin{theorem}\label{theorem.main}
Let $\Gamma \subset G$ be a discrete cocompact group and let $X=\Gamma\backslash \bb{D}$ be the compact hyperbolic surface induced by $\Gamma$.  Let $a_1$, $a_2$ be symbols in the H\"ormander class $S^{\alpha}_{1,0}(T^* X)$ of order $\alpha \leq -6$ and $0$ respectively, which are left-invariant by $\Gamma$, and, when considered as a symbol in the Zelditch calculus in variables $(g,r)\in \Gamma\backslash G \times \bb{R}^+$, have an analytic continuation in $r$ to a strip of width greater than $1/2$.  We further suppose that either $a_1(\cdot, r_j)$ or $a_2(\cdot, r_j)$ has integral zero over $S^*X$ for each $j$ where $\lambda_j=1/4+r_j^2$ is the $j$'th eigenvalue of the Laplace-Beltrami operator on $\Gamma\backslash\bb{D}$.  Let $A_1^{\Gamma}(t)$ and $A_2^{\Gamma}$ be the operators on $L^2(X)$ induced by $A_1(t)=\Op(a_1\circ g^t)$ and $A_2=\Op(a_2)$ using Zelditch quantization. Then $(A_2^{\Gamma})^* A_1^{\Gamma}(t)$ is trace class for each $t$, and 
\begin{equation}
\Trace (A_2^{\Gamma})^* A_1^{\Gamma}(t) 
\end{equation}
decays exponentially as $t \to \pm\infty$. Here $g^t$ is the classical geodesic flow on $S^*\HH^2$, or algebraically, it is the pullback by the right action of the subgroup $A$ on the group $G$.
\end{theorem}

We note that there are interesting examples of symbols which satisfy all the conditions required.  For example, a $\Gamma$-invariant symbol $a(z,b,r)=\varphi(z,b)f(r)$ where $\int_{\Gamma\backslash G} \varphi =0$ and $f$ is a smooth symbol in $r$ which decays on $\bb{R}^+$ at a polynomial rate with an analytic extension to a strip will satisfy all the conditions of Theorem \ref{theorem.main}.

\section{Proofs}

Using Theorem~\ref{thm.helgason.eigenfunction.representation.formula}, we can find an exact formula which relates the quantity $\langle A_1^{\Gamma}\varphi_{j}, A_2^{\Gamma}\varphi_{l}\rangle_{L^2(X)}$ to a certain bilinear form involving the Zelditch symbols $a_1$, $a_2$ of $A_1^{\Gamma}$ and $A_2^{\Gamma}$.

\begin{lemma}\label{lemma.trace.formula}
    Let $A_1^{\Gamma}$ and $A_2^{\Gamma}$ be two pseudodifferential operators on a compact hyperbolic surface $X = \Gamma \backslash \bb{D}$ as in Theorem~\ref{theorem.main}, with $\Gamma$-invariant Zelditch symbols $a_1$ and $a_2$ respectively.  Let $\varphi_{j}$, $\varphi_{k}$ be a pair of Laplace eigenfunctions on $\Gamma\backslash\bb{D}$ with eigenvalues $1/4+r_j^2$ and $1/4+r_k^2$ respectively. We take $r_j, r_k$ to be in $[0, \infty) \cup i [-1/2, 0]$. Then
\begin{equation}
    \langle A_1^{\Gamma} \varphi_{j}, A_2^{\Gamma} \varphi_{l}\rangle_{L^2(X)}=\int_{\Gamma\backslash G} a_1(g,r_j)E_j(g) \left(\int_K \overline{a_2(gk_{\theta},r_l)E_l(gk_{\theta})}  d\theta \right) dg ,
\label{eq:ip1} \end{equation}
    where $E_j(g)$ is the $\Gamma$-invariant distribution on $G$ given by
\begin{equation}
E_j(g) =  e^{(-1/2 + i r_j)\ang{z, b}}   T_j(b). %\quad T_j(db) = T_j(b) db.
\label{eq:Edefn}\end{equation}
Here $E_j$ is the $\Gamma$-invariant distribution mentioned in Corollary \ref{cor:resonance.distribution}.  We use coordinates $(z, b)$ on $G$ as described above, and $T_j(b)$ is the boundary distribution function for $\varphi_j$, making the choice of $r_j$ as above. 
%We have written $T_j(db) = T_j(b) db$, thus $T_j(b)$ is a (distributional) function, while $T_j(db)$ transforms as a density (that is, it transforms as a 1-form, up to sign, under coordinate changes). 
We also remind  the reader that the Haar measure $dg$ is $e^{\ang{z,b}} d\Vol(z) db$ (see \eqref{eq:haar}), accounting for $-1/2$ (instead of $+1/2$) in the exponent    in \eqref{eq:Edefn}. We can interpret $E_j$ as the $j$th Ruelle-Pollicott resonance on $\Gamma \backslash G$ in the first band; see \cite[Section 2]{dya.fau.gui.1}. 
\end{lemma}

\begin{proof}
    We begin with Helgason's theorem, Theorem \ref{thm.helgason.eigenfunction.representation.formula}, on the integral representation for Laplace eigenfunctions. Since $\varphi_{j},\varphi_{l}$ are Laplace eigenfunctions on $\Gamma\backslash\bb{D}$ with eigenvalues $1/4+r_j^2$ and $1/4+r_l^2$ respectively, we can consider them as left $\Gamma$-invariant (or left $\Gamma$-periodic) functions on $\bb{D}$, satisfying $\Delta_{\bb{D}}\varphi_{j}=(1/4+r_j^2)\varphi_{j}$ and $\Delta_{\bb{D}}\varphi_{l}=(1/4+r_l^2)\varphi_{l}$ respectively. Since $\Gamma\backslash\bb{D}$ is compact, they are bounded functions on $\bb{D}$ and hence trivially satisfy the exponential growth bound of Helgason's theorem \ref{thm.helgason.eigenfunction.representation.formula}. Our choice of $r_j, r_l$ ensures that $1/2+ir_j$ and $1/2+ir_l$ don't belong to the exceptional set $0, -1, -2, \ldots$, so we have unique boundary distributions $T_{\varphi_{j}}(b)$ and $T_{\varphi_{l}}(b)$ such that
    \[
    \varphi_{j}(z) = \int_B e^{(\frac{1}{2}+ir_j)\langle z, b\rangle} T_{\varphi_{j}}(b)db, \qquad \varphi_{l}(z) = \int_B e^{(\frac{1}{2}+ir_l)\langle z, b\rangle} T_{\varphi_{l}}(b)db
    \]
    for all $z\in\bb{D}$.
    We then recall that the pseudodifferential operators $A_1=\Op(a_1)$ and $A_2=\Op(a_2)$ defined by Zelditch quantisation act on the eigenfunctions in the following way:
    \begin{equation*}
        \begin{split}
            A_1\varphi_{j}(z)=\int_{B}a_1(z,b,r_j)e^{(\frac{1}{2}+ir_j)\langle z, b\rangle} T_{\varphi_{j}}(b)db\\
            A_2\varphi_{l}(z)=\int_{B}a_2(z,b,r_l)e^{(\frac{1}{2}+ir_l)\langle z, b\rangle} T_{\varphi_{l}}(b)db.
        \end{split}
    \end{equation*}
    In fact, for real $r_j$, $r_l$ this follows directly from \eqref{eq:symbol}. On the other hand, both sides of \eqref{eq:symbol} have an analytic continuation to the strip of radius $1/2$, so the equality persists for $|\Im r| \leq 1/2$. Therefore,  we can express $\langle A_1\varphi_{j}, A_2\varphi_{l}\rangle_{L^2(\Gamma\backslash\bb{D}, dz)}$, for any $j$ and $l$, as
\begin{equation}
    \int_{\Gamma\backslash\bb{D}} \left(\int_{B}a_1(z,b,r_j)e^{(\frac{1}{2}+ir_j)\langle z, b\rangle} T_{\varphi_{j}}(b) \, db \right)  \left( \int_{B}\overline{a_2(z,b',r_l)e^{(\frac{1}{2}+ir_l)\langle z, b'\rangle} T_{\varphi_{l}}(b')} \, db' \right) d\Vol(z).
\label{eq:ip2}\end{equation}
 Now writing $g = (z, b)$, we have $(z, b') = g k_\theta$ for some $\theta = \theta(z, b, b')$, since the right action of $K$ rotates in the fibres of the cosphere bundle $S^* \HH^2$ keeping the basepoint $z$ fixed. In view of equality \eqref{eq:haar}, $dg = e^{\ang{z, b}} d\Vol(z) db = d\Vol(z) dz d\theta$, we have $db = e^{\ang{z, b}} d\theta$ when $z$ is held fixed.    We can therefore express \eqref{eq:ip2} as 
 \begin{equation}
    \int_{\Gamma\backslash G} a_1(g,r_j) E_j(g)   \left( \int_{K}\overline{a_2(g k_\theta,r_l) E_l(g k_\theta)} d\theta \right) dg, 
\label{eq:ip3}\end{equation}
which verifies \eqref{eq:ip1}. 
\end{proof}

\begin{remark}
    We remark that when $A_2$ is the identity operator, $\langle A_1\varphi_{j}, A_2\varphi_{l}\rangle=\langle A_1\varphi_j,\varphi_{l}\rangle$ equals the off-diagonal Wigner distributions $W^{\Gamma}_{j, l}(a_1)$ considered by Anantharaman and Zelditch in \cite{ana.zel.2}, with the diagonal Wigner distributions (with $j=l$) considered previously in \cite{ana.zel.1}. There, they relate these distributions to (families of) eigendistributions of the geodesic flow, called Patterson-Sullivan distributions. They derive an identity involving an intertwining operator which sends Patterson-Sullivan distributions to Wigner distributions. Their derivation of this operator involves expressing $k_{\theta}$ in the right hand side of \eqref{eq:ip3} as a product $\overline{\mathbf{n}}_v\mathbf{a}_t\mathbf{n}_u$ for functions $v,t,u$ of $\theta$ and then using the invariance and eigenvariance properties of the distributions $E_l$ by pullback with respect to the right action of $N$ and $A$ on $G$ respectively. See the second proof of Proposition 5.7 in \cite{ana.zel.2}.
\end{remark}
We now verify the claim \eqref{eq:traceform} in the introduction. 
\begin{corollary}\label{cor.pairing}
Let $w\geq 1/2$ and let $S_{w}$ be the strip $\{ r \in \CC \mid |\Im r| \leq w \}$ in the complex plane of width $w$. 
Let $a_1, a_2$ be symbols as in Theorem~\ref{theorem.main}. Then the trace $A_2^* A_1$ is given by the sesquilinear pairing
\begin{equation}
\ang{a_1,\fT a_2}
\label{eq:pairing}\end{equation}
 of $a_1$ against a distribution $\fT a_2$ where $\fT$ maps symbols to distributions on $\Gamma \backslash G \times S_w$, given by 
\begin{equation}
\fT a_2 (g, r) = \sum_{j} \fT_j a_2(g) \delta_{r_j}(r) , \quad  \fT_j a_2(g) = \overline{E_j(g)} \int_K a_2(g k_\theta, r_l) E_l(g k_\theta) \, d\theta.
\label{eq:fTdefn}\end{equation}
The pairing $\ang{a_1, \fT a_2}$ is well-defined for $a_1 \in S^{-4}_{1, 0}((\Gamma \backslash G \times S_{w}))$. 
\end{corollary}

\begin{proof}
The formula follows immediately from Lemma~\ref{lemma.trace.formula} and the standard formula for the trace
\begin{equation}
\Trace A_2^* A_1 = \sum_{j=0}^\infty \langle A_1 \varphi_j, A_2 \varphi_j \rangle_{L^2(X)}
\end{equation}
using the fact that the $\varphi_j$ are an orthonormal basis for $L^2(X)$. 
We need to check that the sum is well-defined as a distribution. This follows from the estimates below which verify that the pairing is well-defined for $a_1 \in S^{-4}_{1, 0}((\Gamma \backslash G \times S_{w}))$. We use rather crude estimates for this, which accounts for the loss in the order of $a_1$. We are unconcerned by this as our intended application is to operators of order $-\infty$, as in \eqref{eq:qc}. 

Using work of Otal, \cite{Otal}, we find that $T_j$ is the derivative of a function $F_j(b)$ that is H\"older continuous of order $1/2$. The H\"older $1/2$-norm of $F$ is bounded in \cite{Otal} by $C \big( \| \varphi_j \|_\infty + \| \nabla \varphi_j \|_\infty \big)$, where $C$ is independent of $j$. Using standard estimates of eigenfunctions on compact manifolds, for example as proved in \cite{hor.spectral.fn}, gives us an estimate of $C \ang{r_j}^{3/2}$ for the H\"older $1/2$-norm of $F$ and \emph{a fortiori} for its $L^2$ norm. Therefore $T_j$ itself is in $H^{-1}(S^1)$ with a norm estimate $C \ang{r_j}^{3/2}$, with $C$ uniform in $j$. It follows that $E_j$, as a function of $(z, b)$, coordinates which makes sense locally on $\Gamma \backslash G$, is $L^\infty$ in $z$ with values in $H^{-1}(S^1_b)$, with a norm estimate $C \ang{r_j}^{3/2}$.

We now consider the integral over $K$ in the definition \eqref{eq:fTdefn} of $\fT a_2$. Recall that this is an expression for $A_2 \varphi_j$. Since $A_2$ is a pseudodifferential operator of order $0$, this is $L^2(X)$ uniformly in $r$ (with constant depending on a finite number of seminorms of the symbol $a_2$, but not on $j$. 
Viewed as a function on $\Gamma \backslash G$, it is uniformly $L^2$ in $z$, and constant in $b$. The product in \eqref{eq:fTdefn} is therefore $L^2$ in $z$ with values in $H^{-1}$ in $b$, locally, and therefore $H^{-1}$ locally, satisfying 
\begin{equation}
\| \fT_j a_2 \|_{H^{-1}(\Gamma \backslash G)} \leq C \ang{r_j}^{3/2}, \quad \text{with $C$ independent in $j$}.  
\label{eq:H-1}\end{equation}

As for $a_1$, as a symbol of order $-4$ it is in $H^{k}(\Gamma \backslash G)$ for each fixed $r$, with norm bounded by $\ang{r}^{-4}$. The pairing of $a_1$ with the $j$th summand of \eqref{eq:fTdefn} is therefore $O(\ang{r_j}^{-5/2})$, and using Weyl asymptotics, we see that this is summable in $j$. This verifies the claim that the pairing $\ang{\fT a_2, a_1}$ is well-defined for symbols $a_1$ of order $-4$. 

\end{proof}

We can interpret \eqref{eq:pairing}, at least formally, as a sum over $j$ of the classical correlation of the function $a_1(\cdot, r_j)$, with the distribution $\fT_j a_2$. 
Next we check that $\fT_j a_2$ has the required regularity for which we can apply some results on decay of correlations. 

\begin{lemma}\label{lem.FS} The distribution $\fT_j a_2$ in \eqref{eq:fTdefn} is in the anisotropic Sobolev space $\mathcal{H}^{-s}(\Gamma\backslash G)$ described in Section \ref{subsec:aniso} for every $s \geq 1$, and the norm of $\fT_j a_2$ in the space $\mathcal{H}^{-s}(\Gamma\backslash G)$ is bounded by $C \ang{r_j}^{7/2}$. 
\end{lemma}

\begin{proof}
We have already seen, in the proof of Corollary~\ref{cor.pairing}, that $\fT_j a_2$ is in the space $H^{-1}(\Gamma\backslash G)$ with norm bounded by $C \ang{r_j}^{3/2}$. To obtain the strengthened conclusion that  $\fT_j a_2$ in the space $\mathcal{H}^{-s}(\Gamma\backslash G)$, we use the fact that $E_j$ satisfies equations 

\begin{equation}\label{eqn.X+.and.H.smoothness.of.distribution}
    \begin{split}
        X_+ E_j &=0\\
        (H+\frac{1}{2}-ir_j)E_j&=0.
    \end{split}
    \end{equation}
   It follows that $\fT_j a_2 = (A_2 \varphi_j) E_j$ satisfies the equations 
  \begin{equation}\label{eqn.X+.and.H.smoothness.2}
    \begin{split}
        X_+^2 \fT_j a_2 &= E_j (X_+^2 (A_2 \varphi_j)) \\
        (H+\frac{1}{2}-ir_j)^2 \fT_j a_2 &= E_j ((H+\frac{1}{2}-ir_j)^2 (A_2 \varphi_j)).
    \end{split}
    \end{equation}

     To compute the norm of $\fT_j a_2$ in $\mathcal{H}^{-1}(\Gamma\backslash G)$, we choose an elliptic pseudodifferential operator $B$  of variable order $\mathsf{s}$, equal to the variable order of the Sobolev space $\mathcal{H}^{-1}(\Gamma\backslash G)$, together with an elliptic invertible operator $R$  of order $-1$; then an equivalent norm is 
 $$
 \| B \fT_j a_2 \|_{L^2} + \| R \fT_j a_2 \|_{L^2}.
 $$
 Here we note that $\mathsf{s}$ is defined on $T^*(S^* X) = T^*(\Gamma \backslash G)$, takes values in $[-1, 1]$, is equal to $1$ in a conic neighbourhood of the line bundle $E_u^*$ and $-1$ in a conic neighbourhood of $E_s^*$, and is monotone increasing with respect to geodesic flow in direction $t\rightarrow\infty$.
 
 It is immediate that $R \fT_j a_2$ is in $L^2$ with a norm bound of $C \ang{r_j}^{3/2}$, using \eqref{eq:H-1}. To analyze the other term, we observe that, except at the bundle $E_s^*$, one or the other of the operators in \eqref{eqn.X+.and.H.smoothness.2} is elliptic. Using microlocal parametrices for these operators on their respective elliptic sets, we may write
 \begin{equation}
 B = B_1 X_+^2 + B_2 (H+\frac{1}{2}-ir_j)^2 + B_3,
 \label{eq:Edecomp}\end{equation}
 where $B_1$ and $B_2$ have variable order $\mathsf{s}-2$, and $B_3$ has order $-1$ (we may assume without loss of generality that $B_3$ is microsupported in a small conic neighbourhood of $E_s^*$, where $B$ has order $-1$). Then, $B_i \in \Psi^{-1}(S^* X)$ for $i = 1, 2, 3$. Using \eqref{eqn.X+.and.H.smoothness.2} we find that $B \fT_j a_2$ is indeed in $L^2$. (The reason for using the squares of these vector fields in \eqref{eqn.X+.and.H.smoothness.2} was so that we could reduce the orders of $B_1$ and $B_2$ by two relative to $B$, so that they become operators of order $-1$.) Moreover, the right hand side of \eqref{eqn.X+.and.H.smoothness.2} is in $H^{-1}(S^* X)$, with norm bounded by $C \ang{r_j}^{7/2}$, using the same argument as in the previous proof --- the extra powers of $r_j$ arise from up to two derivatives applied to $A_2 \varphi_j$ as well as the term $r_j^2$ in $(H+\frac{1}{2}-ir_j)^2$.  Thus the expressions \eqref{eq:Edecomp} and \eqref{eqn.X+.and.H.smoothness.2} together show that the norm of $B \fT_j a_2$ in $L^2$ is bounded by $\ang{r_j}^{7/2}$. 
 \end{proof}

Next we prove a lemma which shows exponential decay of the correlation of two functions $f_1 \in C^\infty(S^* X)$ and $f_2 \in \mathcal{H}^{-s}(S^* X)$ as $t\rightarrow -\infty$. The argument is relatively standard, but we provide the details for completeness, following the proof of \cite[Corollary 5]{non.zwo.1} rather closely. Note that we use the function space $\mathcal{H}^{-s}$ with $s\geq 1$ since the distributions $E_j$ are in this space.  We will apply Theorem \ref{thm.non.zwo} to the the vector field $-H$ (rather than $H$) which generates the flow $t\mapsto g^{-t}$ for which all the estimates in Theorem \ref{thm.non.zwo} hold with $\mathcal{H}^{-s}$ in place of $\mathcal{H}^s$ and $-H$ in place of $H$.  The reason the estimates still hold under this replacement is simply that the flow $g^{-t}$ reverses the role of the bundles $E_u^*$, $E_s^*$. This will mean that the correlations will decay as $t\rightarrow -\infty$ rather than $t\rightarrow\infty$.

\begin{lemma}\label{lemma.decay.of.cor}
    Let $f_1\in C^{\infty}(S^* X)$ and $f_2 \in \mathcal{H}^{-s}(S^* X)$ be such that one of them has mean zero. Assume that $1 \leq s \leq 2$, and let $0 < \alpha < 1/2$, $N \in \NN, \epsilon>0$ be such that the resolvent is holomorphic for $-\alpha<\Im \lambda<0$ and there is a polynomial bound on the growth of the resolvent of $P = iH$ on the anisotropic space $\mathcal{H}^{-s}$ in the region $\Im \lambda >-\alpha, |\lambda| >\epsilon $ as in Theorem \ref{thm.non.zwo}:
\begin{equation}
\| (P - \lambda)^{-1} \|_{\H^{-s} \to \H^{-s}} \leq C \ang{\lambda}^N, \quad \Im \lambda >-\alpha, |\lambda|>\epsilon.
\end{equation}
Then there is a constant $C_\alpha$ such that for all $t>0$,
\begin{equation}
    \int_{\Gamma\backslash G}f_1(ga_{-t})\overline{f_2(g)} dg \leq C_\alpha e^{-\alpha t} \| f_1\| _{H^{N+4}} \| f_2\| _{\mathcal{H}^{-s}}.
\end{equation}
\end{lemma}

\begin{remark} Before we prove this lemma, we remark that Faure-Tsujii have stronger results. From their works 
 \cite{fau.tsu.4} and \cite{fau.tsu.5}, it seems from their proof of the band structure of the Pollicott-Ruelle spectrum, that they get a \emph{uniform} estimate of the resolvent in the strip $\Im \lambda \geq -\alpha$ (away from the pole at $\lambda = 0$) on certain anisotropic Sobolev spaces $\widetilde{\mathcal{H}}^s$ that are slightly different from those considered here. If this is so then an abstract result in semigroup theory, the Gearhart-Pr\"{u}ss-Greiner theorem, see page 302 in \cite{engel.nagel}, shows that the semigroup is exponentially decaying, that is, that 
 $$\|e^{tH}f\|_{\widetilde{\mathcal{H}}^s}\leq Ce^{-\alpha t}\|f\|_{\widetilde{\mathcal{H}}^s}$$
 for $t > 0$. This implies an improvement of the result claimed in the lemma, that the correlation is bounded by 
     \[
    \langle e^{tH}f_1, f_2\rangle \leq \|e^{tH}f_1\|_{\widetilde{\mathcal{H}}^s} \|f_2\|_{\widetilde{\mathcal{H}}^{-s}}
    \]

 The proof we present of Lemma \ref{lemma.decay.of.cor} below is a cruder proof sufficient for our purposes where we only require that a spectral gap for the resolvent $(iH-\lambda)^{-1}$ exists with a polynomial growth bound on the resolvent asymptotically in the spectral gap.
\end{remark}

\begin{proof}
%We recall that the space $\H^s$ is defined in \cite{non.zwo.1} as $e^{-s\G} L^2$, where $\G$ is a pseudodifferential operator with a symbol of logarithmic growth at fibre-infinity. By definition of the space $\H^s$, the operator $e^{-s\G} : L^2 \to \H^s$ is unitary, with inverse $e^{s\G} : \H^s \to L^2$. 

Observe that, by the density of $L^2 \cap \H^{-s}$ in $\H^{-s}$, it is sufficient to prove the estimate for $f_2 \in L^2 \cap \H^{-s}$. Assuming this, we write the correlation using $L^2$-spectral theory for the self-adjoint operator $P = iH$  on $L^2(S^* X)$ with domain $D(P) = \{ f \in L^2 \mid Pu \in L^2 \}$. We will also pedantically write $P_s$ for $iH$ on the space $\H^{-s}$ with domain $D(P_s)=\{f\in \mathcal{H}^{-s}\mid Pu\in\mathcal{H}^{-s}\}$.
%and we note that $P_{s}$ is unitarily equivalent to the operator $\Psg$ on $L^2$ defined by  
%\begin{equation}
%\Psg = e^{-s\G} P e^{s\G}. 
%\label{eq:Psg}\end{equation}
%Moreover we have $m(\Psg) = e^{-s\G} m(P) e^{s\G}$ for any Borel function $m$. 

Applying $L^2$-spectral theory for the self-adjoint operator $P$ and the functions $f_1, f_2 \in L^2$, we have 
\begin{equation}
\langle e^{itP}f_1,  f_2\rangle_{L^2(\Gamma\backslash G)} = \int_{\bb{R}} e^{it\lambda} d \langle   f_1,  E(-\infty, \lambda) f_2\rangle 
\label{eq:dc1}\end{equation}
where $E(I)$ is the spectral projector for $P$ onto the set $I \subset \RR$. 
To gain a decay factor in this integral, we exploit the fact that $ f_1$ is smooth. We let $\overline{f_1} = (P + i)^{N+2}  f_1$, which is $C^\infty$, and has mean zero if $f_1$ does, since $\ang{P^j f_1, 1} = \ang{f_1, P^j 1} = 0$.  We can express the correlation as
\begin{equation}
\langle e^{itP} f_1,  f_2\rangle_{L^2(\Gamma\backslash G)} = 
    \int_{\bb{R}} \frac{e^{it\lambda} }{(\lambda + i)^{(N+2)}} d \langle \overline{f_1}, E(-\infty, \lambda)  f_2\rangle . 
\label{eq:dc2}\end{equation}

Next we use Stone's formula to express the spectral measure in terms of the resolvent. We can write \eqref{eq:dc2} as 
\begin{multline}
 \lim_{\epsilon \searrow 0}        \frac{1}{2\pi i} \int_{\bb{R}}\frac{e^{it\lambda}}{(\lambda+i)^{(N+2)}} \Big\langle \overline{f_1}, (P-\lambda-i\epsilon)^{-1}  f_2 \Big\rangle \, d\lambda \\
            -  \lim_{\epsilon \searrow 0}   \frac{1}{2\pi i} \int_{\bb{R}} \frac{e^{it\lambda}}{(\lambda+i)^{(N+2)}} \Big\langle \overline{f_1}, (P-\lambda+i\epsilon)^{-1}  f_2 \Big\rangle \, d\lambda .
\label{eq:dc3}    \end{multline}
The second term is zero, as we see by shifting the contour of integration to $\Im \lambda = c$, $c < 0$ and sending $c \to -\infty$. For this we simply need the estimate $\| (P - \lambda)^{-1} \|_{L^2 \to L^2} \leq |\Im \lambda|^{-1}$. 

To deal with the first term, where we cannot shift the contour to $\Im \lambda < 0$ because of the spectrum of $P$ along the real line, we pass to the operator $P_{s}$, as we may since $f_2 \in \H^{-s}$. We thus have 
 \begin{equation}
  \langle e^{itP} f_1,  f_2\rangle_{L^2(\Gamma\backslash G)} =     
        \frac{1}{2\pi i} \int_{\bb{R}} \frac{e^{it\lambda}}{(\lambda+i)^{(N+2)}} \Big\langle \overline{f_1}, (P_s-\lambda+i\epsilon)^{-1}  f_2 \Big\rangle \, d\lambda .
\label{eq:dc4}    \end{equation}
The resolvent of $P_s$ has a meromorphic continuation to the half-space $\Im \lambda > -1/2$, with a pole at the origin, \cite{dya.fau.gui.1}. The pole is simple with residue being the projection on to constants. Since either $\overline{f_1}$ or $f_2$ has mean zero, the inner product in \eqref{eq:dc4} is holomorphic across zero so the contour can be pushed down to $\Im \lambda = -\alpha$, using the decay factor $(\lambda + i)^{-(N+2)}$ which overcomes the $\ang{\lambda}^N$ growth of the resolvent. Doing so picks up a decay factor $e^{-\alpha t}$ and we obtain the required estimate by estimating 
$$
\Big| \Big\langle \overline{f_1}, (P_s-\lambda - i\alpha)^{-1}  f_2 \Big\rangle \Big| \leq C \ang{\lambda}^{N} \| \overline{f_1} \|_{\H^{s}} \| f_2 \|_{\H^{-s}} \leq C \ang{\lambda}^{N} \| f_1 \|_{H^{N+4}} \| f_2 \|_{\H^{-s}},
$$ 
and integrating in $\lambda$. The last inequality follows as $H^2$ embeds into $\H^{s}$ since $1 \leq s \leq 2$.  
\end{proof}

Now we are in a position to prove the main theorem. 

\begin{proof}[Proof of Theorem~\ref{theorem.main}] A pseudodifferential operator on a compact manifold is trace class provided it has order strictly less than minus the dimension of the manifold. Since we have assumed that $A_1$ has order $\leq -6$ and $A_2$ has order zero, the evolved operator $A_1(t)$ also has order $\leq -6$, so $A_2^* A_1(t)$ has order $\leq -6$ and is trace class. 

We first prove the exponential decay of the trace in the limit $t \to -\infty$. 

By Corollary~\ref{cor.pairing}, the trace of $A_2^*A_1(t)$ is given by a sum
$$
\sum_j \ang{g_t^*a_1( \cdot , r_j), \fT_j a_2} 
$$
of distributions $\fT_j a_2$ depending on $a_2(\cdot, r_j)$ applied to $g_t^*a_1( \cdot , r_j)$. 
We first consider an individual term in this sum. By Lemma~\ref{lem.FS}, the distribution $\fT_j a_2$ is in the anisotropic Sobolev space $\H^{-s}$ for $s \geq 1$, with norm in this space $O(\ang{r_j}^{7/2})$. Provided that $a_1$ is a symbol of order $-6$ or below, the norm of $a_1(\cdot, r_j)$ in any standard Sobolev space $H^m$ is $O(\ang{r_j}^{-6})$, from the symbol estimates, showing that each term in the sum is $O(\ang{r_j}^{-5/2})$. Using Weyl asymptotics we see that $r_j$ is bounded above and below by a multiple of $\sqrt{j}$ as $j \to \infty$, hence the sum is absolutely convergent. Moreover, according to Lemma~\ref{lemma.decay.of.cor}, each term in the sum is bounded by 
$$
C_\alpha e^{-\alpha t} \| a_1(\cdot, r_j)\| _{H^{N+4}} \| \fT_j a_2 \| _{\mathcal{H}^{-s}}
$$
for some $0<\alpha < 1/2$. Removing the exponentially decaying factor in time, we can sum the series, leading 
to the conclusion that the trace decays exponentially in time. 

We next briefly discuss the case $t \to \infty$. Observe first that there is nothing in the statement of the theorem to suggest that one direction of time is favoured over another. Examining the proof, we see that in considering `plane waves' of the form $e^{(1/2 + ir)\ang{z, b}}$, we made a choice to use plane waves that are constant on \emph{forward} horocycles. One can equally well consider plane waves that are constant on \emph{backward} horocycles; these are just the previous plane waves composed with the inversion map, which is the group element 
$$
\begin{pmatrix} 0 & 1 \\ -1 & 0 \end{pmatrix} \in \PSL, 
$$
that is, the element of the subgroup $K$ that rotates by $\pi$ in the fibres. If we do that, then we find that the corresponding Ruelle-Pollicott resonances $\hat E_j$, as in \eqref{eq:Edefn}, are invariant under the generator $X_-$ of unstable horocycle flow rather than the generator $X_+$ of stable horocycle flow as is the case for $E_j$. Correspondingly, the $\hat E_j$ are in the opposite anisotropic space, that is regular at the stable line bundle $E_s^*$ and rough at the unstable bundle $E_u^*$. This leads to exponential decay in as $t \to \infty$. We omit the details.

\end{proof}

\bibliographystyle{plain}
\bibliography{references}
\nocite{*}

\end{document}